\documentclass[a4paper, notitlepage]{amsart}
\usepackage{amsmath, amssymb, amsthm, graphicx, xspace,url,a4wide}

\newcommand{\Z}{\ensuremath{\mathbb{Z}}\xspace}
\newcommand{\Q}{\ensuremath{\mathbb{Q}}\xspace}
\newcommand{\R}{\ensuremath{\mathbb{R}}\xspace}
\newcommand{\C}{\ensuremath{\mathbb{C}}\xspace}
\newcommand{\A}{\ensuremath{\mathbb{A}}\xspace}
\newcommand{\F}{\ensuremath{\mathbb{F}}\xspace}

\newcommand{\MaxSpec}{\ensuremath{\mathrm{MaxSpec}}\xspace}
\newcommand{\rSpec}{\ensuremath{\underline{\mathrm{Spec}}}\xspace}

\newcommand{\m}{\ensuremath{\mathfrak{m}}\xspace}

\newcommand{\mbar}{\ensuremath{\overline{\mathfrak{m}}}\xspace}
\newcommand{\rhobar}{\ensuremath{\overline{\rho}}\xspace}

\newcommand{\OO}{\ensuremath{\mathcal{O}}\xspace}

\newcommand{\T}{\ensuremath{\mathbb{T}}\xspace}

\newcommand{\comment}[1]{}
\newcommand{\GL}{\ensuremath{\mathrm{GL}}\xspace}

\DeclareMathOperator{\Gal}{Gal}

\DeclareMathOperator{\Hom}{Hom}

\newtheorem{theorem}{Theorem}
\newtheorem{proposition}[theorem]{Proposition}
\newtheorem{corollary}[theorem]{Corollary}
\newtheorem{lemma}[theorem]{Lemma}
\newtheorem{definition}[theorem]{Definition}

\newtheorem*{theor}{Theorem}
\newtheorem*{remark}{Remark}
\input xy
\xyoption{all}
\begin{document}

\title{Level raising and completed cohomology}
\author{James Newton}
\email{jjmn2@cam.ac.uk}
\begin{abstract}
We describe an application of Poincar{\'e} duality for completed homology spaces (as defined by Emerton) to level raising for $p$-adic modular forms. This allows us to give a new description of the image of Chenevier's $p$-adic Jacquet-Langlands map between an eigencurve for a definite quaternion algebra and an eigencurve for $\GL_2$. The points on the eigencurve at which we `raise the level' are (non-smooth) points of intersection between an `old' and a `new' component.
\end{abstract}
\maketitle
\bibliographystyle{amsalpha}

\section{Introduction}
In this note we describe an application of Poincar{\'e} duality for completed homology spaces (as defined by Emerton) to level raising for $p$-adic modular forms. Fix a prime $l \ne p$ and a positive integer $N$ coprime to $pl$. Let $D(N)$ and $D(N,l)$ be the reduced eigencurves of tame level $\Gamma_1(N)$ and $\Gamma_1(N)\cap\Gamma_0(l)$ respectively, constructed using the Hecke operators away from $N$ (we \emph{do} use the Hecke operators at $l$). Let $D(N,l)_c$ be the equidimensional closed rigid analytic subvariety of $D(N,l)$ whose points are those with irreducible attached residual Galois representation $\rhobar : \Gal(\overline{\Q}/\Q) \rightarrow \GL_2(\overline{\F}_p)$, and denote by $D(N)_c$ the analogous subvariety of $D(N)$. Define $D(N,l)^{\mathrm{new}}_c$ to be the the closed rigid analytic subvariety of $D(N,l)_c$ whose points come from overconvergent modular forms which vanish under the two natural trace maps from forms of level $\Gamma_1(N)\cap\Gamma_0(l)$ to forms of level $\Gamma_1(N)$. Our main result is the following, proved in section \ref{mainthm}:

\begin{theor}
$D(N,l)^{\mathrm{new}}_c$ is equidimensional of dimension $1$. It is equal to the Zariski closure of the classical $l$-new points in $D(N,l)_c$.
\end{theor}

This answers a question raised by the main theorem of \cite{MR2111512}, where the image of a $p$-adic Jacquet-Langlands map is identified as the Zariski closure of the classical $l$-new points in $D(N,l)$. At least in the subspace $D(N,l)_c$, the above theorem allows a point in this image to be identified by a natural condition on the corresponding overconvergent modular form.

This theorem is equivalent to a level raising result on the eigencurve. Our investigation of this question was motivated by a conjecture made by Paulin, prompted by results on local-global compatibility for the eigencurve in his thesis \cite{Pa}. For each point $x$ of $D(N)_c$ there are $2$ points $x_1, x_2$ of $D(N,l)_c$ corresponding to the two roots of the $l$th Hecke polynomial for $x$. Suppose $T_l^2(x)-(l+1)^2 S_l(x)=0$. Equivalently, one of the two points $x_1, x_2$, say $x_1$, is in $D(N,l)^{\mathrm{new}}_c$. The above theorem implies that there is a $1$-dimensional irreducible component of $D(N,l)^{\mathrm{new}}_c$ passing through $x_1$. This irreducible component is an irreducible component of $D(N,l)_c$, with the property that every classical point on it is $l$-new. The point $x_1$ is not smooth, since two irreducible components (an old one and a new one) intersect there. 

This note is a complement to the paper \cite{chicomp}, where an analogous theorem is proved for the eigencurve corresponding to overconvergent automorphic forms on a definite quaternion algebra - in fact in that setting our results include the points where the attached residual Galois representation is reducible. This difference arises because in \cite{chicomp} a form of Ihara's lemma in characteristic $0$ is used, whereas in this paper we make use of the classical Ihara's lemma to show the injectivity of a level raising map between completed cohomology spaces. We make use of completed cohomology (and homology) because Poincar\'{e} duality provides us with a natural pairing between completed homology spaces, which are finitely generated modules under a non-commutative local ring. Constructing a suitable pairing between spaces of overconvergent modular forms, the approach taken in \cite{chicomp}, seems to be rather more difficult in the more geometric setting of overconvergent modular forms.

Similar results have been obtained by Paulin \cite{NonComp}, using a different approach which relies on the main result of \cite{Emlg} (and hence on Colmez's $p$-adic local Langlands correspondence). 

\section{Completed cohomology}
We begin by defining completed homology and cohomology spaces, following Emerton's seminal paper \cite{emint} and Calegari and Emerton's survey \cite{emcalsumm}.
Fix a finite extension $E$ of $\Q_p$, with ring of integers $\OO$ and uniformiser $\varpi$. Let $K^p$ be a compact open subgroup of $\GL_2(\A_f^p)$. If $K_p$ is a compact open subgroup of $G:=\GL_2(\Q_p)$ then we denote by $Y(K_p K^p)$ the open modular curve (over $\C$) of level $K_p K^p$ and write
$$H^i(K^p,A):= \varinjlim_{K_p} H^i(Y(K_p K^p),A),$$ where the limit is over all compact open $K_p$, and $A$ is either $E$, $\OO$ or $\OO/\varpi^s$ for some $s > 0$. Note that there is a natural $G$ action on $H^i(K^p,A)$, and a natural $G$-equivariant isomorphism
$$H^i(K^p,\OO)/\varpi^s H^i(K^p,\OO) \cong H^i(K^p,\OO/\varpi^s).$$
\begin{definition}\label{ccdef}We define the completed $i$th cohomology space of tame level $K^p$ to be $$\widehat{H}^i(K^p,\OO):= \varprojlim_s H^i(K^p,\OO)/\varpi^sH^i(K^p,\OO)\cong \varprojlim_s H^i(K^p,\OO/\varpi^s).$$ 
\end{definition}
Applying the analogue of this construction to compactly supported cohomology, we obtain the space $\widehat{H}^i_c(K^p,\OO).$ We can also form completed homology spaces in a dual manner:
\begin{definition}\label{chdef}The completed $i$th homology space of tame level $K^p$ is $$\widehat{H}_i(K^p,\OO):= \varprojlim_{K_p} H_i(Y(K_p K^p),\OO).$$
\end{definition} 
The same construction with Borel-Moore homology gives $\widehat{H}^{\mathrm{BM}}_i(K^p,\OO)$. We write $G_0$ for $\GL_2(\Z_p)\subset G$. The key finiteness property for the spaces we have defined above is contained in the following result:
\begin{proposition}\label{firstdual}
\begin{enumerate}
\item $\widehat{H}_i(K^p,\OO)$ and $\widehat{H}^{\mathrm{BM}}_i(K^p,\OO)$ are finitely generated left modules under the natural action of $\OO[[G_0]]$. The topology induced by the $\OO[[G_0]]$-module structure is equivalent to the projective limit topology on these spaces.
\item We have $\Hom_{\mathrm{cts}}(\widehat{H}^{\mathrm{BM}}_1(K^p,\OO),\OO) \cong \widehat{H}^1_c(K^p,\OO)$ and $\Hom_{\mathrm{cts}}(\widehat{H}^1_c(K^p,\OO),\OO) \cong \widehat{H}^{\mathrm{BM}}_1(K^p,\OO)$.
\end{enumerate}
\end{proposition}
\begin{proof}
The first part of the  proposition follows from the first part of Theorem 1.1 in \cite{emcalsumm}. As for the second part, Proposition 4.3.6 of \cite{emint} shows that $\widehat{H}^2_c(K^p,\OO)=0$, whilst viewing Borel-Moore homology as the homology of the compactified modular curve relative to the cusps, it is easy to see that $\widehat{H}_0^\mathrm{BM}(K^p,\OO)=0$. Our proposition now follows from the third part of Theorem 1.1 in \cite{emcalsumm}.
\end{proof}
\begin{remark}Theorem 1.1 of \cite{emcalsumm} also describes the connection between $\widehat{H}_i(K^p,\OO)$ and $\widehat{H}^i(K^p,\OO)$, which is slightly more complicated than in the compactly supported case, since $\widehat{H}^0(K^p,\OO)$ is non-zero.\end{remark}
\begin{definition} The Hecke algebra $\T(K^p)$ is defined to be the weakly closed $\OO$-algebra of $G$-equivariant endomorphisms of $\widehat{H}^1(K^p,\OO)$ topologically generated by Hecke operators $T_q, S_q$ for primes $q \nmid p$ that are unramified in $K^p$.
\end{definition}
\section{Non-optimal levels}
We now work with two fixed tame levels, $$K(N) := \{ g \in \GL_2(\widehat{\Z}^p) : g \equiv \begin{pmatrix}
\ast & \ast \\ 0 & 1
\end{pmatrix} \mod N\}$$  and $$K(N,l) := K(N)\cap \{ g \in \GL_2(\widehat{\Z}^p) : g \equiv \begin{pmatrix}
\ast & \ast \\ 0 & \ast
\end{pmatrix} \mod l\}.$$ We have a continuous morphism $\T(K(N,l)) \rightarrow \T(K(N))$, with image $\T(K(N))^{(l)}$ the subalgebra of $\T(K(N))$ generated by the Hecke operators away from $l$. The following definition can be found in section 1.2 of \cite{brem}. 
\begin{definition}An ideal $I$ of $\T(K(N))^{(l)}$ is \emph{Eisenstein} if the map $\lambda: T(K(N))^{(l)} \rightarrow T(K(N))^{(l)}/I$ satisfies $\lambda(T_q) = \epsilon_1(q) + \epsilon_2(q)$ and $\lambda(qS_q) = \epsilon_1(q)\epsilon_2 (q)$ for all $q \nmid plN$, for some characters $\epsilon_1,\epsilon_2 : \Z_{plN}^\times \rightarrow T(K(N))^{(l)}/I$. \end{definition}

Fix a maximal ideal $\m$ of $\T(K(N))^{(l)}$ which is not Eisenstein. We can pull back $\m$ to a maximal ideal $\mbar$ of $\T(K(N,l))$ (maximal since $\mbar$ will certainly contain $p$). 

\begin{proposition}
The natural maps $\alpha: \widehat{H}^1_c(K(N),\OO) \rightarrow \widehat{H}^1(K(N),\OO)$ and $\beta: \widehat{H}^1_c(K(N,l),\OO) \rightarrow \widehat{H}^1(K(N,l),\OO)$ become isomorphisms after localising at the non-Eisenstein maximal ideals $\m$ and $\mbar$ respectively. The same holds for the maps from usual to Borel-Moore homology.\end{proposition}\begin{proof}Proposition 4.3.9 of \cite{emint} shows that the maps $\alpha$ and $\beta$ are surjective. Let $\widehat{M}$ denote the kernel of $\alpha$. Note that $\widehat{M}$ is a torsion free $\OO$-module, since $\widehat{H}^1_c(K(N),\OO)$ is a torsion free $\OO$-module. Corollaire 3.1.3 of \cite{brem} shows that $\widehat{M}$ is  `Eisenstein', i.e. if there is a system of Hecke eigenvalues $\lambda: T(K(N))^{(l)} \rightarrow L$, $L$ a finite extension of $E$, with $M_L^\lambda \neq 0$, then $\ker(\lambda)$ is an Eisenstein ideal. 

Note that $\widehat{M}$ contains a dense subspace $M:= \ker(H^1_c(K(N),\OO) \rightarrow H^1(K(N),\OO))$. In fact $\widehat{M}$ is equal to the $\varpi$-adic completion of $M$, one way to see this is using the explicit description of $\widehat{M}$ (and $M$) given by Proposition 3.1.1 of \cite{brem}. The space $M\otimes_{\OO} \overline{E}$ is spanned by Hecke eigenvectors, which have Eisenstein systems of eigenvalues, so $(M\otimes_{\OO} \overline E)_\m=0$ and hence $M_\m=0$, since $M$ is $\OO$-torsion free. Now $\widehat{M}_\m$ is a direct summand of $\widehat{M}$, and the projection $\widehat{M} \rightarrow \widehat{M}_\m$ is $\varpi$-adically continuous, with $M$ in the kernel, hence this projection is the zero map and $\widehat{M}_\m = 0$. So $\alpha$ does give an isomorphism after localising at $\m$. 

The same argument applies to the map $\beta$. The statement for homology follows from the duality of the second part of Proposition \ref{firstdual} (note that this duality holds for usual homology and cohomology after localising at a non-Eisenstein maximal ideal, since this kills $\widehat{H}^0(K^p,\OO)$) by the same argument we used to show that $\widehat{M}_\m = 0$.
\end{proof}
\begin{remark}In the proof of Proposition 7.7.13 of \cite{Emlgc} the fact that localising at non-Eisenstein maximal ideals induces an isomorphism between compactly supported completed cohomology and completed cohomology is used. The author thanks Matthew Emerton for communicating the above argument for why this follows from Corollaire 3.1.3 of \cite{brem}.
\end{remark}

Since the Hecke algebras are semilocal, localising at the maximal ideals $\m$ and $\mbar$ gives us direct summands of the original (co)homology spaces. Note that the second part of Proposition \ref{firstdual} now gives us dualities between $\widehat{H}^1(K(N),\OO)_\m$ (respectively $\widehat{H}^1(K(N,l),\OO)_{\mbar}$) and $\widehat{H}_1(K(N),\OO)_\m$ (respectively $\widehat{H}_1(K(N,l),\OO)_{\mbar}$).

For each compact open $K_p \subset G$ there are two degeneracy maps from $Y(K_pK(N,l))$ to $Y(K_pK(N))$, which in the limit give rise to a natural level raising map $$i : \widehat{H}^1(K(N),\OO)_\m^2 \rightarrow \widehat{H}^1(K(N,l),\OO)_{\mbar}.$$ 
There is a map in the other direction $$i^\dagger : \widehat{H}^1(K(N,l),\OO)_{\mbar} \rightarrow \widehat{H}^1(K(N),\OO)_\m^2$$ given by the inverse system of maps $$i^\dagger_s : H^1_c(K_pK(N,l),\OO/\varpi^s) \rightarrow H^1_c(K_pK(N),\OO/\varpi^s)^2$$ where $i^\dagger_s$ is the adjoint under Poicar\'{e} duality of the usual level raising map (as described in section 3 of \cite{DT} for the Shimura curve case). We can form the $\OO$-duals of these maps to get maps $i^*, i^{\dagger *}$ between homology spaces. 

A standard calculation shows that the composition $i^\dagger \cdot i$ acts by the matrix $\begin{pmatrix}
l+1 & T_l \\ S_l^{-1}T_l & l+1
\end{pmatrix}$ on $\widehat{H}^1(K(N),\OO)_\m^2$.

\begin{definition}
We define the $l$-new space to be $$\widehat{H}^1_{\mathrm{new}}(K(N,l),\OO)_{\mbar} := \mathrm{ker}(i^\dagger).$$ 
\end{definition}
By duality, we can identify the dual $\Hom_\mathrm{cts}(\widehat{H}^1_{\mathrm{new}}(K(N,l),\OO)_{\mbar},\OO)$ with $$\widehat{H}_1^{\mathrm{new}}(K(N,l),\OO)_{\mbar} := \widehat{H}_1(K(N,l),\OO)_{\mbar}/i^{\dagger *}(\widehat{H}_1(K(N),\OO)_{\m}^2).$$

There is a form of Ihara's lemma in this situation, which follows easily from the classical Ihara's lemma:
\begin{lemma}\label{ihsurj}
The map $i$ is an injection. Dually, the map $i^*$ is a surjection.
\end{lemma}
\begin{proof}
It is enough to show that the induced map $$i : \widehat{H}^1(K(N),\OO)_\m^2/\varpi^s \rightarrow \widehat{H}^1(K(N,l),\OO)_{\mbar}/\varpi^s$$ is an injection for all $s$. Let $\T$ be the abstract Hecke algebra over $\OO$ generated by $T_q, S_q$ for primes $q \nmid plN$. It is clear that $\m$ and $\mbar$ pull back to the same non-Eisenstein maximal ideal $\mathfrak{M}$ of $\T$. Since $\widehat{H}^1(K(N),\OO)_\m/\varpi^s \cong H^1(K(N),\OO)_\mathfrak{M}/\varpi^s$ and $\widehat{H}^1(K(N,l),\OO)_{\mbar}/\varpi^s \cong H^1(K(N,l),\OO)_\mathfrak{M}/\varpi^s$ we just need to show that the map $$i : H^1(K(N),\OO)_\mathfrak{M}^2 \rightarrow H^1(K(N,l),\OO)_\mathfrak{M}$$ is injective with torsion free cokernel, which is the usual Ihara's lemma as in \cite{MR804706}.
\end{proof}

Poincar\'{e} duality at finite levels also gives rise to a duality between completed homology spaces:
\begin{proposition}\label{poin}
There are $\OO[[K_p]]$-module isomorphisms \begin{eqnarray*}\delta_1:\Hom(\widehat{H}_1(K(N),\OO)_\m,\OO[[K_p]])& \cong &\widehat{H}_1(K(N),\OO)_\m \\
 \delta_2:\Hom(\widehat{H}_1(K(N,l),\OO)_{\mbar},\OO[[K_p]]) &\cong &\widehat{H}_1(K(N,l),\OO)_{\mbar},\end{eqnarray*} where $K_p$ is any open subgroup of $G_0$.
\end{proposition}
\begin{proof}This follows from localising the Poincar\'{e} duality spectral sequence that can be found in \cite{emcalsumm}.
\end{proof}
\begin{lemma}\label{newdual}
We have a commutative diagram
$$\xymatrix{\Hom(\widehat{H}_1(K(N),\OO)_\m^2,\OO[[K_p]]) \ar[r]^{^t (i^*)}\ar[d]_{\sim}^{\delta_1 \oplus \delta_1} & \Hom(\widehat{H}_1(K(N,l),\OO)_{\mbar},\OO[[K_p]])\ar[d]_{\sim}^{\delta_2}\\
\widehat{H}_1(K(N),\OO)_\m^2\ar[r]^{i^{\dagger *}} & \widehat{H}_1(K(N,l),\OO)_{\mbar}
},$$ where the map $^t (i^*)$ is the $\OO[[K_p]]$-dual of $i^*$.
\end{lemma}
\begin{proof} It suffices to check that for each $K_p'$ an open normal subgroup of $K_p$, the map induced by $^t (i^*)$ on the finite level quotient $H_1(K(N)K_p',\OO)_\m^2$ is equal to the map induced by $i^{\dagger *}$. This follows from the adjointness of $i$ and $i^\dagger$ under classical Poincar\'{e} duality, and the fact that the duality of Proposition \ref{poin} is induced by the maps 
$$\Hom_{\OO[K_p/K_p']}(H_1(K(N)K_p',\OO),\OO[K_p/K_p']) \cong \Hom_{\OO}(H_1(K(N)K_p',\OO),\OO) \rightarrow H_1^{\mathrm{BM}}(K(N)K_p',\OO),$$
where the first map is a canonical isomorphism of $\OO[K_p/K_p']$-modules (as in Lemma 6.1 of \cite{emhida}) and the second map is given by classical Poincar\'{e} duality. 
\end{proof}

Say that a compact open subgroup $K_f$ of $\GL_2(\A_f)$ is \emph{neat} if $\GL_2(\Q)$ acts on $(\C \backslash \R) \times \GL_2 (\A_f )/K_f$ without fixed
points. The following proposition is due to Emerton, to appear in \cite{Emlg}:
\begin{proposition}\label{free}
 For a compact open $K_p \subset G$ which is pro-$p$, such that $K_pK(N)$ and $K_pK(N,l)$ are neat (equivalently $K_pK(N)$ is neat), $\widehat{H}_1(K(N),\OO)_\m$ and $\widehat{H}_1(K(N,l),\OO)_{\mbar}$ are free $\OO[[K_p]]$-modules.
\end{proposition}

\begin{theorem}\label{lrfree} 
If $K_p$ satisfies the conditions of the above proposition, then $\widehat{H}_1^{\mathrm{new}}(K(N,l),\OO)_{\mbar}$ is also a free $\OO[[K_p]]$-module.
\end{theorem}
\begin{proof}
By lemma \ref{ihsurj}, there is a short exact sequence
$$\xymatrix{
0\ar[r]&\mathrm{ker}(i^*)\ar[r]&\widehat{H}_1(K(N,l),\OO)_{\mbar}\ar[r]^{i^*}&\widehat{H}_1(K(N),\OO)_\m^2\ar[r]&0
}.$$
Since $\widehat{H}_1(K(N),\OO)_\m^2$ is free, there is a section to $i^*$, and so $\mathrm{ker}(i^*)$ is a direct summand of the free module $\widehat{H}_1(K(N,l),\OO)_{\mbar}$. Hence $\mathrm{ker}(i^*)$ is projective. But we assumed $K_p$ pro-$p$, so $\OO[[K_p]]$ is local implying (by the noncommutative version of Nakayama's lemma) that $\mathrm{ker}(i^*)$ is free. Now applying the functor $\Hom(-,\OO[[K_p]])$ to the above short exact sequence of free modules, we get a diagram
\scriptsize $$\xymatrix{
0\ar[r]&\Hom(\widehat{H}_1(K(N),\OO)_\m^2,\OO[[K_p]])\ar[r]\ar[d]^{\sim} &\Hom(\widehat{H}_1(K(N,l),\OO)_{\mbar},\OO[[K_p]])\ar[r]\ar[d]^{\sim} &\Hom(\mathrm{ker}(i^*),\OO[[K_p]])\ar[r]&0\\
&\widehat{H}_1(K(N),\OO)_\m^2 \ar[r]^{i^{\dagger *}}& \widehat{H}_1(K(N,l),\OO)_{\mbar}
},$$\normalsize
where the first row is short exact, the vertical maps are the isomorphisms provided by Proposition \ref{poin} and the horizontal map on the second row is identified as $i^{\dagger *}$ by Lemma \ref{newdual}. Hence this diagram gives an isomorphism of $\OO[[K_p]]$-modules between the free module $\Hom(\mathrm{ker}(i^*),\OO[[K_p]])$ and $\mathrm{coker}(i^{\dagger *})=\widehat{H}_1^{\mathrm{new}}(K(N,l),\OO)_{\mbar}$.
\end{proof}
\section{Eigencurves of newforms}\label{mainthm}
We now recall Emerton's eigenvariety construction (see sections 2.3 and 4 of \cite{emint}), which we will apply to the space $\widehat{H}^1_{\mathrm{new}}(K(N,l),E)_{\mbar}$, with the slight variation that we will use Hecke operators at $l$ as well. Applying the Jacquet functor (see \cite{MR2292633}) to the locally analytic vectors in $\widehat{H}^1_{\mathrm{new}}(K(N,l),E)_{\mbar}$ we get an essentially admissible $T(\Q_p)$ representation, where $T$ is a maximal torus of $\GL_2$. This corresponds to a coherent sheaf $\mathcal{E}$ on the rigid analytic space $\widehat{T}$ which classifies continuous characters of $T(\Q_p)$. Let $\T_l$ be the abstract Hecke algebra over $\OO$ generated by $T_q, S_q$ for primes $q \nmid plN$ and the operator $U_l$. The algebra $\T_l$ acts on $\mathcal{E}$, generating a coherent sheaf of algebras $\mathcal{A}$ on $\widehat{T}$. We have the relative spectrum $\widetilde{D}^{\mathrm{new}}:=\rSpec(\mathcal{A})$ of $\mathcal{A}$ over $\widehat{T}$, a rigid analytic space over $E$. We define $\widetilde{D}^{\mathrm{new}}$ to be the reduced rigid analytic space $\rSpec(\mathcal{A})^{\mathrm{red}}$.

We have the following corollary of Theorem \ref{lrfree}:

\begin{corollary}
The space $\widetilde{D}^{\mathrm{new}}$ is equidimensional of dimension $2$.
\end{corollary} 
\begin{proof}
Fix a pro-$p$ compact open subgroup of $G$, which is small enough so that $K_pK(N)$ is neat. Now Theorem \ref{lrfree} implies that $\widehat{H}_1^{\mathrm{new}}(K(N,l),\OO)_{\mbar}$ is a free $\OO[[K_p]]$-module, i.e. isomorphic to $\OO[[K_p]]^r$ as an $\OO[[K_p]]$-module for some integer $r$. For ease of notation, let $V$ denote the $G$-representation $\widehat{H}^1_{\mathrm{new}}(K(N,l),E)_{\mbar}$. By duality (and then inverting $p$), $V$ is isomorphic to $\mathcal{C}(K_p,E)^r$ as a $K_p$-representation, where $\mathcal{C}(K_p,E)$ is the space of continuous functions from $K_p$ to $E$, viewed as a representation of $K_p$ by letting $K_p$ act by right translation on functions. This shows that the space of locally analytic vectors $V_{la}$ is isomorphic to $\mathcal{C}^{la}(K_p,E)^r$ as a $K_p$ representation, where $\mathcal{C}^{la}(K_p,E)$ is the space of locally analytic functions from $K_p$ to $E$. We let $T_0$ denote the compact subgroup $T\cap K_p$ of $T$, and write $M$ for the strong dual of the Jacquet module $J_B(V_{la})'_b$, noting that $M$ is the space of global sections of the sheaf $\mathcal{E}$ over $\widehat{T}$. The rigid analytic variety $\widehat{T}_0$ parameterises the characters of $T_0$, and there is a natural map $\widehat{T} \rightarrow \widehat{T}_0$ induced by restriction of characters. 

The proof of Proposition 4.2.36 in \cite{MR2292633} shows that, writing $\widehat{T}_0$ as a union of admissible affinoid subdomains $\MaxSpec(A_n)$, there is an isomorphism of $E\{\{z,z^{-1}\}\}\hat{\otimes}_E A_n$-modules $$\label{fredholm}M_n:=M\hat{\otimes}_{\mathcal{C}^{an}(\widehat{T}_0,E)}A_n \cong E\{\{z,z^{-1}\}\}\hat{\otimes}_{E[z]}(W_n\hat{\otimes}_{E}A_n),$$ where $W_n$ is an $E$-Banach space and $z$ is a certain element of $T$ acting as a compact operator on $W_n\hat{\otimes}_{E}A_n$. Let $Y$ be the subgroup of $T$ generated by $z$, with corresponding rigid analytic character variety $\widehat{Y}$. We can now describe how to apply the machinery of \cite{Bu2,Chenun} to get our desired equidimensionality result. Let $\mathcal{E}_n$ be the pullback of $\mathcal{E}$ to $\widehat{T}\times_{\widehat{T}_0}A_n$, let $\mathcal{A}_n$ be the pullback of $\mathcal{A}$ to $\widehat{T}\times_{\widehat{T}_0}A_n$, and let $Z_n \hookrightarrow \widehat{Y}\times_E A_n$ be the Fredholm variety cut out by the characteristic power series of $z$ acting on $W_n\hat{\otimes}_{E}A_n$. Since $\mathcal{E}_n$ is the sheaf associated to $M_n$, there is a finite map $\rSpec(\mathcal{A}_n) \rightarrow Z_n$. In the language of \cite{Bu2}, $Z_n$ is the `spectral variety', and $\MaxSpec(A_n)$ is `weight space'.

An admissible cover $\mathcal{C}$ of $Z_n$ is constructed in section 4 of \cite{Bu2}, so that for $X \in \mathcal{C}$, with image $V$ in $\MaxSpec(A_n)$ under projection, the pullback of $\rSpec(\mathcal{A}_n)$ to $Y$ is given by the spectrum of a commutative algebra of endomorphisms (coming from the Hecke operators and the action of $T$) of a locally free, finite type $\OO(V)$-module. Now Lemme 6.2.10 of \cite{Chenun} applies, in the same way as the proof of Proposition 6.4.2 in \cite{Chenun}, to deduce that $\rSpec(\mathcal{A}_n)$ is equidimensional of dimension $2$, for all $n$, hence $\rSpec(\mathcal{A})$ is equidimensional of dimension $2$.
\end{proof}

It follows from Proposition 4.4.6 in \cite{emint} that $\widetilde{D}^{\mathrm{new}}$ is a product of an equidimension $1$ space $D^{\mathrm{new}}$ and weight space, by left exactness of the Jacquet functor and the fact that $\widehat{H}^1_{\mathrm{new}}(K(N,l),E)_{\mbar}$ is stable under the action of $\mathcal{C}(T_0,E)$ induced by the map labelled 4.2.5 in \cite{emint}. This fact holds since this twisting action is induced from a $\GL_2(\A_f)$-equivariant action on the direct limit of completed cohomology spaces over all tame levels, so it will preserve the kernel of $i^\dagger$ (which was defined by means of degeneracy maps).  
Note that we can also apply the same construction to $\widehat{H}^1(K(N,l),E)_{\mbar}$ to get an equidimension $1$ space $D$.
 
We now return to the notation of the introduction. For each non-Eisenstein maximal ideal $\mbar$ of $\T(K(N,l))$ let $D(N,l)_{\mbar}$ be the equidimensional closed rigid subspace of $D(N,l)_c$ whose points have residual Galois representation corresponding to the Hecke character induced by $\mbar$. Set $D(N,l)^{\mathrm{new}}_{\mbar}=D(N,l)^{\mathrm{new}}_c \cap D(N,l)_{\mbar}$. There is a natural isomorphism between $D(N,l)_{\mbar}$ and $D$, since both spaces are reduced, equidimensional and contain a Zariski dense set of points corresponding to the systems of Hecke eigenvalues arising from classical modular forms of tame level $\Gamma_1(N)\cap\Gamma_0(l)$ (see the proof of Theorem 7.5.8 in \cite{Emlgc}). Similarly, since $D^{\mathrm{new}}$ is equidimensional, it is isomorphic to the Zariski closure of the classical $l$-new points in $D(N,l)_{\mbar}$, so there is a closed embedding $D^{\mathrm{new}} \hookrightarrow D(N,l)^{\mathrm{new}}_{\mbar}$. We can now prove the theorem stated in the introduction.

\begin{theorem}
$D(N,l)^{\mathrm{new}}_c$ is equidimensional of dimension $1$. It is equal to the Zariski closure of the classical $l$-new points in $D(N,l)_c$.
\end{theorem}
\begin{proof}
It is enough to show that the closed embedding $D^{\mathrm{new}} \hookrightarrow D(N,l)^{\mathrm{new}}_{\mbar}$ is an isomorphism (for each non-Eisenstein $\mbar$). As in Proposition 4.7 of \cite{MR2111512} the complement of the image of this embedding is a union of irreducible components of dimension $0$. Suppose the complement $\mathcal{Z}$ is non-empty, containing a point $x_1$. This point must lie over a point $x$ of $D(N)_\m$, since in $D(N,l)_{\mbar}$ it lies in an irreducible component of dimension $1$ which only contains $l$-old classical points, since $x_1 \in \mathcal{Z}$. We can now compare the completed cohomology side and the overconvergent modular form side, since we can characterise `old and new' points like $x_1$ using Hecke operators - they come from points $x$ satisfying $T_l(x)^2-(l+1)^2S_l(x)=0$. On the cohomology side this corresponds to an eigenclass in the kernel of $i^\dagger \cdot i$ (recall that this composition can be expressed in terms of Hecke operators), so in particular $x_1$ corresponds to an eigenclass in the kernel of $i^\dagger$ - therefore it lies in  $D^\mathrm{new}$, contradicting the assumption that $x_1$ lies in the complement.
\end{proof}

\section*{Acknowledgements}The author is supported by an EPSRC doctoral training grant, and would like to thank Matthew Emerton for helpful correspondence, as well as communicating details of \cite{Emlg}, and Kevin Buzzard for many useful conversations.

\bibliography{dissbib}
\end{document}